\documentclass{amsart}
\usepackage{amssymb}
\usepackage{amsmath}
\usepackage{tikz-cd}
\usepackage{amsthm}
\theoremstyle{plain}

 \newcommand{\rr}{\mathbb R}

\theoremstyle{plain}
\newtheorem{thm}{Theorem}
\newtheorem{prop}{Proposition}

\newtheorem{coro}{Corollary}

\usepackage{geometry}
\theoremstyle{definition}
\newtheorem{exam}{Example}
\newtheorem{rem}{Remark}

\begin{document}
\setcounter{page}{1}

\title{ Dunford-Pettis and Compact Operators Based on Unbounded Absolute Weak Convergence}
\author[Nazife Erkur\c{s}un-\"Ozcan, Niyazi An{\i}l Gezer, Omid Zabeti ]{Nazife Erkur\c{s}un-\"Ozcan$^{(1)}$, Niyazi An{\i}l Gezer$^{(2)}$, Omid Zabeti$^{(3, *)}$}

\address{$^{1}$ Department of Mathematics, Faculty of Science, Hacettepe University, Ankara, 06800,Turkey.}
\email{{erkursun.ozcan@hacettepe.edu.tr}}

\address{$^{2}$ Department of Mathematics, Middle East Technical University, Ankara, 06800, Turkey.}
\email{{ngezer@metu.edu.tr }}

\address{$^{3}$ Department of Mathematics, Faculty of Mathematics, University of Sistan and Baluchestan, P.O. Box: 98135-674, Zahedan, Iran.}
\email{{o.zabeti@gmail.com}}



\date{Received: xxxxxx; Revised: yyyyyy; Accepted: zzzzzz.
\newline \indent $^{*}$ Corresponding author}

\begin{abstract}
In this paper, using the concept of unbounded absolute weak convergence ($uaw$-convergence, for short) in a Banach lattice, we define two classes of continuous operators, named $uaw$-Dunford-Pettis and $uaw$-compact operators. We investigate some properties and relations between them. In particular, we consider some hypotheses on domain or range spaces of operators such that the adjoint or the modulus of a $uaw$-Dunford-Pettis or $uaw$-compact operator inherits a similar property. In addition, we look into some connections between compact operators, weakly compact operators, and Dunford-Pettis ones with $uaw$-versions of these operators. Moreover, we examine some relations between $uaw$-Dunford-Pettis operators, $M$-weakly compact operators, $L$-weakly compact operators, and $o$-weakly compact ones. As a significant outcome, we show that the square of any positive $uaw$-Dunford-Pettis ($M$-weakly compact) operator on an order continuous Banach lattice is compact. Many examples are given to illustrate the essential conditions, as well.
\end{abstract} \maketitle





\section{Introduction and Preliminaries}

The notion of $uo$-convergence under the name individual convergence was initially introduced in \cite{N} and "$uo$-convergence" was proposed firstly in \cite{DM}. Recently, various types of interesting papers about $uo$-convergence in Banach lattices have been announced by several authors (see \cite{G, GTX, GX} for more expositions on these results). $Un$-convergence was introduced by Troitsky in \cite{T} and further investigated in \cite{DOT, KMT}. Unbounded convergent nets in term of weak convergence, $uaw$-convergence, was introduced by Zabeti and considered in \cite{Z}.

Let $E$ be a Banach lattice. For a net $(x_\alpha)$ in $E$, if there is a net $(u_\gamma)$, possibly over a
different index set, with $u_\gamma \downarrow 0$ and for every $\gamma$ there exists $\alpha_0$ such
that $|x_{\alpha} - x| \leq u_\gamma$ whenever $\alpha \geq \alpha_0$, we say that $(x_\alpha)$ converges to $x$ in order, in notation, $x_\alpha \xrightarrow{o}x$.  A net $(x_{\alpha})$ in $E$ is said to be unbounded order convergent ( $uo$-convergent, in brief) to $x \in E$ if for each $u \in E_{+}$, the net $(|x_{\alpha} - x| \wedge u)$ converges to zero in order.  It is called unbounded norm convergent ($un$-convergent, for short) if $\||x_{\alpha}-x|\wedge u\|\rightarrow 0.$  For a version of an unbounded convergent net in term of weak convergence, a net $(x_{\alpha})$ in a Banach lattice $E$ is said to be unbounded absolutely weakly convergent to $x\in E$ if for each positive $u\in E$, one has $|x_{\alpha}-x|\wedge u\xrightarrow{w}0.$  In a recent paper \cite{Z}, several properties of $uaw$-convergence have been investigated. In particular, order continuous Banach lattices and reflexive ones are characterized in terms of $uaw$-convergent nets. In addition, it is shown that the $uaw$-convergence is topological.

In this note, by an operator, we mean a bounded operator between Banach lattices, unless otherwise explicitly stated.

It is known that compact operators have important applications both in analysis and other disciplines. In this paper, the concept of a $uaw$-compact operator is defined. An operator $T\colon X\rightarrow E$, where $X$ is a Banach space and $E$ is a Banach lattice, is said to be (sequentially) $uaw$-compact if $T(B_X)$ is relatively (sequentially) $uaw$-compact where $B_X$ denotes the closed unit ball of the Banach space $X$. Equivalently, for every bounded net $(x_\alpha)$  (respectively, every bounded sequence $(x_n)$) its image has a subnet (respectively, subsequence), which is $uaw$-convergent. We further say that the operator $T$ is $un$-compact if $T(B_X)$ is relatively $un$-compact in $E$. In \cite{KMT}, some properties of $un$-compact operators are studied.

Moreover, we consider a $uaw$-version of Dunford-Pettis operators. For the general theory of Dunford-Pettis operators, reader is referred to \cite{AB, M, S}. Suppose $E$ is again a Banach lattice and $X$ is a Banach space. We say that an operator $T\colon E\rightarrow X$  is $uaw$-Dunford-Pettis  if for every norm bounded sequence $(x_n)$ in $E$, $x_n\xrightarrow{uaw}0$ implies $||T(x_n)||\rightarrow 0$.

In the present paper, we investigate relationships between compact and Dunford-Pettis operators in the $uaw$-version. Some properties of $uaw$-compact and $uaw$-Dunford-Pettis operators are studied. Moreover, we utilize some conditions on domain or range of operators to ensure us when the adjoint or the modulus of a $uaw$-compact or $uaw$-Dunford-Pettis operator has the same property. Furthermore, we employ some assumptions to establish some connections between $uaw$-Dunford-Pettis operators, $M$-weakly compact operators, $L$-weakly compact operators, and $o$-weakly compact ones. As a main consequence, we deduce that the square of a positive $uaw$-Dunford-Pettis ($M$-weakly compact) operator on an order continuous Banach lattice is compact. In addition, various examples are given to make the concepts and hypotheses more understandable.

Denote by $B_{UDP}(E),B_{DP}(E),K_{uaw}(E),K_{un}(E)$ the spaces of all $uaw$-Dunford-Pettis, Dunfor-Pettis, $uaw$-compact and $un$-compact operators on a Banach lattice $E$, respectively.
For other necessary terminology on vector and Banach lattice, we refer the reader to \cite{AA, AB}.

\section{Main Results}
\begin{prop}\label{2}
 Suppose that $E$ is a Banach lattice whose dual space is order continuous and $X$ is a Banach space. Then, every Dunford-Pettis operator $T\colon E\rightarrow X$ is $uaw$-Dunford-Pettis.
\end{prop}
\begin{proof}
Suppose $T\in B_{DP}(E,X)$ and $(x_n)$ is a norm bounded sequence in $E$ which is $uaw$-convergent to zero. By \cite[Theorem 7]{Z}, it is weakly convergent. By the assumption, $\|T(x_n)\|\rightarrow{0}$, as desired.
\end{proof}
Note that order continuity of $E'$ is essential in Proposition \ref{2} and can not be dropped. To see this, consider the identity operator $I$ on $\ell_1$. It follows from the Schur property of $\ell_1$ that $I$ is Dunford-Pettis. However it can not be $uaw$-Dunford-Pettis as the $uaw$-null sequence $(e_i)$ formed by the standard basis of $\ell_1$ is not norm convergent to zero. In addition, it can be easily seen that every $uaw$-Dunford-Pettis operator is automatically continuous but the converse is not true, in general; again, consider the identity operator on $\ell_1$.

\begin{rem} \label{7000}
Suppose that $E$ is an $AM$-space and $X$ is a Banach space. Since the lattice operations in $E$ are weakly sequentially continuous \cite[Theorem 4.31]{AB} and in view of Proposition \ref{2}, it can be seen that an operator $T:E\rightarrow X$ is $uaw$-Dunford-Pettis if and only if it is Dunford-Pettis. Suppose further that $E$ is an atomic order continuous Banach lattice. It follows from \cite[Proposition 2.5.23]{Ni} that if an operator $T\colon E\rightarrow X$ is $uaw$-Dunford-Pettis, then it is a Dunford-Pettis operator.
\end{rem}
It is known that every compact operator is Dunford-Pettis. In the following example, we show that in the case of a $uaw$-Dunford-Pettis operator, the situation is different.
\begin{exam} Let $T\colon \ell_1\rightarrow \rr$ be defined by $T((x_n))=\sum\limits_{n=1}^{\infty}x_n$ for every $(x_n)\in\ell_1$. Since $T$ is of finite rank, it is compact. It follows by considering the standard basis of $\ell_1$ that $T$ can not be a $uaw$-Dunford-Pettis operator.
\end{exam}
A typical example of a Dunford-Pettis operator which is not compact is the identity operator on $\ell_1$ because of the Schur property. But this operator does not do the job for the $uaw$-case since it is not also $uaw$-Dunford-Pettis. Nevertheless, there is a good news if one considers the Lozanovsky-like example as it is described in \cite[Page 289, Exercise 10]{AB}.
\begin{exam} \label{1}
Consider the operator $T\colon C[0,1]\rightarrow c_0$ given by
\begin{center}
$T(f)=(\int_0^1 f(t)\sin t \: dt,\int_0^1 f(t)\sin 2t \: dt,\ldots )$
\end{center}
for every $f\in C[0,1]$. It can be easily seen that $T$ is not order bounded so that by \cite[Theorem 5.7]{AB}, $T$ is not compact. Denote by $(f_n)\subseteq  C[0,1]$ a norm bounded sequence for which $f_n\xrightarrow{uaw}0$ holds.
It follows from \cite[Theorem 7]{Z} that $f_n\xrightarrow{w} 0$ and that $||T(f_n)||=\sup_{m\geq 1}|\int_0^1f_n(t)\sin mt \: dt|\leq\int_0^1|f_n(t)|dt\rightarrow 0.$ Hence, the noncompact operator $T$ is an $uaw$-Dunford-Pettis operator.
\end{exam}

As in \cite[Proposition 9.1]{KMT} and using \cite[Theorem 4 and Proposition 14]{Z}, we have the same conditions for $uaw$-compactness and sequentially $uaw$-compactness of an operator.
\begin{prop}\label{1000}
Let  $T\colon E\rightarrow F$  be an operator
 between Banach lattices.
\begin{itemize}
\item[(i)] If $F$ is order continuous and  has a quasi-interior point then $T$ is $uaw$-compact iff it is sequentially $uaw$-compact;
\item[(ii)] If $F$ is order continuous and $T$ is $uaw$-compact then $T$ is sequentially $uaw$-compact;
\item[(iii)] If $F$ is an atomic KB-space then $T$ is $uaw$-compact and sequentially $uaw$-compact.
\end{itemize}
\end{prop}
\begin{rem}
The fact which is used in proof of \cite[Proposition 9.1, (i)]{KMT} is that $un$-topology on a Banach lattice $E$ is metrizable if and only if $E$ has a quasi-interior point. This result can be restated in term of $uaw$-topology provided that $E$ is order continuous. Note that order continuity is essential and can not be dropped; consider $E=\ell_{\infty}$. It is easy to see that $uaw$-topology and absolute weak topology agree on $B_E$. But $B_E$ is not weakly metrizable since $E^{'}$ is not separable. This implies that $E$ can not be metrizable with respect to the $uaw$-topology.
\end{rem}


Let us continue with several ideal properties.

\begin{prop} Let $S\colon E\rightarrow F$ and $T\colon F\rightarrow G$  be two operators between Banach lattices.
\begin{itemize}
\item[\em (i)] {If $T$ is (sequentially) $uaw$-compact and $S$ is continuous then $TS$ is (sequentially) $uaw$-compact}.
\item[\em (ii)] {If $T$ is a $uaw$-Dunford-Pettis operator and $S$ is either (sequentially) $un$-compact or $uaw$-compact then $TS$ is compact}.
\item[\em (iii)]
{ If $T$ is $uaw$-Dunford-Pettis and $S$ is Dunford-Pettis then $TS$ is Dunford-Pettis}.
\item[\em (iv)]{ If $T$ is continuous and $S$ is $uaw$-Dunford-Pettis, then $TS$ is $uaw$-Dunford-Pettis}.
\end{itemize}
\end{prop}
\begin{proof}
$(i)$ We prove the results for the sequence case. For nets, the proof is similar. Suppose $(x_n)\subseteq E$ is a bounded sequence. By the assumption, the sequence $(S(x_n))$ is also norm bounded.  Therefore, there is a subsequence $TS(x_{n_k})$ which is $uaw$-convergent.

$(ii)$ Suppose $(x_n)$ is a bounded sequence in $E$. There is a subsequence $(x_{n_k})$ such that $S(x_{n_k})\xrightarrow{uaw}x$ for some $x\in F.$ Thus, by the hypothesis, $\|T(S(x_{n_k}))-T(S(x))\|\rightarrow 0$, as desired.

$(iii)$ Suppose $(x_n)$ is a sequence in $E$ which is weakly null. By the assumption, $\|S(x_n)\|\rightarrow 0$. It follows that $S(x_n)\xrightarrow{uaw}0$. Again, this implies that $\|TS(x_n)\|\rightarrow 0$.

$(iv)$ Suppose $(x_n)$ is a norm bounded sequence in $E$ which is $uaw$-null. By the hypothesis, $\|S(x_n)\|\rightarrow 0$ so that $\|T(S(x_n))\|\rightarrow 0$, as desired.
\end{proof}
\begin{coro}
Suppose $E$ is a Banach lattice. Then $B_{UDP}(E)$ is a subalgebra of $B(E)$.
\end{coro}
In general, we have $K(E)\subseteq K_{un}(E)\subseteq K_{uaw}(E).$ In the next discussion, we show that not every $uaw$-compact operator is $un$-compact.

\begin{exam} The inclusion $\ell_2\hookrightarrow \ell_{\infty}$ is weakly compact by \cite[Theorem 5.24]{AB}. Hence, it is sequentially $uaw$-compact because range of the operator is an $AM$-space. However it is not sequentially $un$-compact. Since by \cite[Theorem 2.3]{KMT}, it should be compact which is not possible.
\end{exam}

\begin{rem} $K_{un}(E)$ and $K_{uaw}(E)$ are not order closed in the usual order of the space of all continuous operators on $E$, as shown by \cite[Example 9.3]{KMT}; see also \cite[Theorem 4]{Z}.
\end{rem}

Following results are motivated by the Krengel's Theorem, see \cite[Theorem 5.9]{AB}.

\begin{thm}\label{3}
If $E$ is an $AL$-space and $F$ is a Banach lattice whose dual space is order continuous. Then every sequentially $uaw$-compact operator $T$ from $E$ into $F$ has a sequentially $uaw$-compact adjoint.
\end{thm}
\begin{proof} Let $T\colon E\rightarrow F$ be a sequentially $uaw$-compact operator. For every norm bounded sequence $(x_n)$  in $E$, the sequence $T(x_n)$  has a subsequence $T(x_{n_k})$ which is convergent in the $uaw$-topology. By \cite[Theorem 7]{Z}, the subsequence is weakly convergent. This implies that the operator $T$ is weakly compact. By the Gantmacher's theorem \cite[Theorem 5.23]{AB}, it follows that $T'$ is weakly compact. Since range of $T'$ is an $AM$-space, it is sequentially $uaw$-compact.
\end{proof}
\begin{rem}
Note that order continuity of $F'$ is essential and can not be removed. Consider the identity operator on $\ell_1$. One may verify that it is $uaw$-compact; for $\ell_1$ is an atomic $KB$-space, therefore using \cite[Theorem 7.5]{KMT} and \cite[Theorem 4]{Z}, yield the desired result. But its adjoint is the identity operator on $\ell_{\infty}$ which is not sequentially $uaw$-compact.
\end{rem}
\begin{thm}\label{CM1} If $E$ is an $AL$-space and $F$ is a reflexive Banach lattice. Then every order bounded sequentially $uaw$-compact operator $T$ from $E$ into $F$, has a weakly compact modulus.
\end{thm}
\begin{proof} By Theorem \ref{3}, if $T$ is sequentially $uaw$-compact then $T'$ is a sequentially $uaw$-compact operator. Note that $E'$ is an $AM$-space. So, the operator $T'$ is weakly compact and the result follows from  \cite[Theorem 5.35]{AB}.
\end{proof}

\begin{prop} Let $E$ be a Banach lattice whose dual space is atomic and order continuous. Also let $F$ be a Banach lattice whose dual is order continuous. Then, every (sequentially) $un$-compact operator  $T\colon E\rightarrow F$  has a (sequentially) $un$-compact adjoint operator $T'\colon F'\rightarrow E'$.
\end{prop}
\begin{proof} For any norm bounded sequence $(x_n)$ in $E$, the sequence $(T(x_n))$ has a subsequence which is $un$-convergent to zero by $un$-compactness. By \cite[Theorem 6.4]{DOT}, it is weakly convergent. Hence, the operator $T$ is weakly compact.  It follows from Gantmacher's theorem that $T'$ is weakly compact. By \cite[Proposition 4.16]{KMT}, the operator $T'$ is $un$-compact.
\end{proof}

Recall that, see \cite{AB} for details, an operator $T\colon E\rightarrow F$ is $M$-weakly compact if for every norm bounded disjoint sequence $(x_n)$ one has $||Tx_n||\rightarrow 0$; $T$ is $L$-weakly compact if every disjoint sequence $(y_n)$ in the solid hull of $T(B_E)$ is norm null; see \cite{BW} for a recent progress on this topic.
\begin{prop} \label{5}
If $T\colon E\rightarrow F$ is a $uaw$-Dunford-Pettis operator then $T$ is $M$-weakly compact; in particular, it is weakly compact. Also, if $F'$ is order continuous and $T\colon E\rightarrow F$ is a $uaw$-compact operator, then $T$ is weakly compact.
\end{prop}
\begin{proof}  If $(x_n)$ is norm bounded disjoint sequence in $E$, by \cite[Lemma 2]{Z}, $x_n\xrightarrow{uaw} 0$. Hence, $||Tx_n||\rightarrow 0$.
The second part follows from \cite[Theorem 7]{Z}.
\end{proof}
For the converse, we have the following.
\begin{thm}\label{*}
Suppose $E$ and $F$  are Banach lattices such that either $E$ or $F$ is order continuous. Then every positive $M$-weakly compact operator from $E$ into $F$ is $uaw$-Dunford-Pettis.
\end{thm}
\begin{proof}
Suppose $(x_n)$ is a bounded positive $uaw$-null sequence in $E$ and $\varepsilon>0$ is arbitrary. By \cite[Theorem 5.60]{AB} due to Meyer-Nieberg, there is a positive $u\in E$ with $\|T(x_n)-T(x_n\wedge u)\|<\frac{\varepsilon}{2}$. First, suppose $E$ is order continuous; since $x_n\wedge u\xrightarrow{w}0$ and the sequence is order bounded, by \cite[Theorem 4.17]{AB}, we conclude that  $\|x_n\wedge u\|\rightarrow 0$ so that $\|T(x_n\wedge u)\|\rightarrow 0$. Now, assume $F$ is order continuous; $x_n\wedge u\xrightarrow{w}0$ results in $T(x_n\wedge u)\xrightarrow{w}0$. Note that this sequence is order bounded so that by \cite[Theorem 4.17]{AB}, $\|T(x_n\wedge u)\|\rightarrow 0$. In any case, we see that $\|T(x_n)\|<\varepsilon$ for sufficiently large $n$, as claimed.
\end{proof}
\begin{coro}\label{***}
Suppose either $E$ or $F$ is order continuous. Then every $L$-weakly compact lattice homomorphism from $E$ to $F$ is $uaw$-Dunford-Pettis.
\end{coro}
\begin{proof}
It can be verified easily that $T$ is $M$-weakly compact ( for example see \cite[Page 337, Exercise 4]{AB}). Now, the conclusion follows from Theorem \ref{*}.
\end{proof}
\begin{rem}
Suppose $E$ and $F$ are Banach lattices. An operator $T:E\to F$ is said to be $uaw$-continuous if it maps bounded $uaw$-null sequences to $uaw$-null ones. It can be verified that every $uaw$-Dunford-Pettis operator is $uaw$-continuous but the converse is not true, in general; consider the identity operator on $\ell_1$. $L$-weakly compact operators are fruitful tools for the following.
\end{rem}
\begin{thm}\label{**}
Suppose $E$ is a Banach lattice and $F$  is an order continuous Banach lattice. Then every $L$-weakly compact $uaw$-continuous operator from $E$ into $F$ is $uaw$-Dunford-Pettis.
\end{thm}
\begin{proof}
Suppose $(x_n)$ is a bounded positive $uaw$-null sequence in $E$ and $\varepsilon>0$ is arbitrary. By \cite[Theorem 5.60]{AB}, there is a positive $u\in F$ with $\||T(x_n)|-|T(x_n)|\wedge u\|<\frac{\varepsilon}{2}$. Since $T(x_n)\xrightarrow{uaw}0$, we see that $|T(x_n)|\wedge u\xrightarrow{w}0$. Note that this sequence is order bounded so that by \cite[Theorem 4.17]{AB}, $\||T(x_n)|\wedge u\|\rightarrow 0$. Therefore, $\|T(x_n)\|<\varepsilon$ for sufficiently large $n$, as claimed.
\end{proof}
In the following example, we show that adjoint of a $uaw$-Dunford-Pettis operator need not be $uaw$-Dunford-Pettis.
\begin{exam}\label{2500}
Consider the operator $T$ given in Example \ref{1}. We claim that its adjoint is not $uaw$-Dunford-Pettis. The adjoint $T':\ell_1\rightarrow M[0,1]$ is defined via $T'((x_n))(f)=\Sigma_{n=1}^{\infty}x_n(\int_{0}^{1}f(t)\sin nt dt)$, in which $M[0,1]$ is the space of all regular Borel measures on $[0,1]$. Note that the standard basis $(e_n)$ is $uaw$-null. For each $n\in\Bbb N$, put $f_n(t)=\sin nt$. Then we have
\[\|T'(e_n)\|\ge \|T'(e_n)(f_n)\|=\int_{0}^{1}(\sin nt)^2 dt \nrightarrow 0.\]
\end{exam}
\begin{rem}
Observe that Example \ref{2500} can be employed to show that positivity assumption in Theorem \ref{*} and $uaw$-continuity hypothesis in Theorem \ref{**} are essential and can not be removed. The operator $T'$ is not positive; since $T$ is $uaw$-Dunford-Pettis, it is $M$-weakly-compact. By \cite[Theorem 5.67]{AB}, $T'$ is also $M$-weakly compact but as we see, it is not $uaw$-Dunford-Pettis. Furthermore, again, using \cite[Theorem 5.67]{AB} convinces us that $T'$ is also $L$-weakly compact. We claim that $T'$ is not $uaw$-continuous. Note that $e_n\xrightarrow{uaw}0$. For every $n\in \Bbb N$, consider $f_n(t)=\sin nt$. Also, since the sequence $(\sin n)$ is dense in $[-1,1]$, we can choose sufficiently large $n\in \Bbb N$ with $\sin n>\frac{1}{4}$. Suppose $\delta_{1}$ is the Dirac measure at point $x_0=1$. Then, $(T'(e_n)\wedge \delta_{1})(\sin nt)>\frac{1}{4}$.
\end{rem}
Now, recall that an operator $T:E\to X$ from a Banach lattice to a Banach space is $o$-weakly compact if it carries order intervals to weakly relatively compact sets. Compatible with \cite[Theorem 5.91 and Corollary 5.92]{AB} and \cite[Lemma 2]{Z}, one may verify the following.
\begin{prop}\label{1001}
Every $uaw$-Dunford-Pettis operator $T:E\to X$ from a Banach lattice to a Banach space is $o$-weakly compact.
\end{prop}
\begin{prop}\label{1002}
The square of a $uaw$-Dunford-Pettis operator carries order intervals into norm totally bounded sets.
\end{prop}
Now, we have the following.
\begin{thm}
Suppose $E$ is a Banach lattice and $T$ is a positive $uaw$-Dunford-Pettis operator on it. Then, for every positive operator $S$ dominated by $T^2$, $S^2$ is compact.
\end{thm}
\begin{proof}
By Proposition \ref{5} and Proposition \ref{1001}, $T$ is $o$-weakly compact and $M$-weakly compact. Moreover, by Proposition \ref{1002}, $T^2$ maps order intervals into norm totally bounded sets. Now, the conclusion follows from \cite[Page 338, Exercise 13]{AB}.
\end{proof}
Observe that since the identity operator on $\ell_1$ is Dunford-Pettis, we can not expect compactness of any power of $T$; the following is surprising.
\begin{coro}\label{6000}
Suppose $E$ is a Banach lattice. Then, for every positive $uaw$-Dunford-Pettis operator $T$ on $E$, $T^4$ is a compact operator.
\end{coro}
\begin{coro}
Suppose $E$ is an order continuous Banach lattice and $T$ is a positive $uaw$-Dunford-Pettis operator on $E$. Then, for every operator $S$, with $0\leq S\le T$, $S^2$ is compact. In particular, the square of a positive $uaw$-Dunford-Pettis operator is compact.
\end{coro}
\begin{proof}
By Proposition \ref{1001}, $T$ is $o$-weakly compact. This means that the order bounded set $T[0,x]$ is relatively weakly compact. By \cite[Theorem 4.17]{AB}, it is relatively compact. Now, from \cite[Page 338, Exercise 13]{AB}, we conclude that every positive operator $S$ dominated by $T$, has a compact square.
\end{proof}

\begin{coro}
Suppose $E$ is a Banach lattice. Then the identity operator on $E$ is $uaw$-Dunford-Pettis if and only if $E$ is finite dimensional.
\end{coro}
\begin{proof}
Suppose the identity operator on $E$ is $uaw$-Dunford-Pettis. By Corollary \ref{6000}, it is compact. This yields that $E$ is finite dimensional. suppose $E$ is a finite dimensional Banach lattice. So, it is atomic and reflexive. Therefore, every $uaw$-null sequence is weakly null so that norm null. This means that the identity operator on $E$ is $uaw$-Dunford-Pettis.
\end{proof}
Furthermore, considering Theorem \ref{*}, we get the following important result.
\begin{coro}
Suppose $E$ is an order continuous Banach lattice. Then, the square of every positive $M$-weakly compact operator on $E$ is compact.
\end{coro}
Similar to the case of usual compact operators and Dunford-Pettis ones, it might seem at the first glance that every $uaw$-compact operator is $uaw$-Dunford-Pettis; the following example is surprising.

\begin{exam} The inclusion $\ell_2\hookrightarrow \ell_{\infty}$ is weakly compact by \cite[Theorem 5.24]{AB}. Previously, we showed that this operator is sequentially $uaw$-compact. However it is not $uaw$-Dunford-Pettis. For, the standard basis $(e_n)$ is $uaw$-null but it is not norm convergent to zero.
\end{exam}
Also, the other implication may fail, as well.
\begin{exam} Consider the inclusion map $J\colon L^{\infty}[0,1]\rightarrow L^1[0,1]$. It follows from \cite[Page 313, Exercise 7]{AB} that $J$ is weakly compact. In fact, $J$ is $uaw$-Dunford-Pettis. To see this, suppose $(f_n)$ is a norm bounded sequence which converges to zero in the $uaw$-topology, by \cite[Theorem 7]{Z}, it follows that it is weakly convergent. Since $L^1[0,1]\subseteq (L^{\infty}[0,1])'$ and the constant function one
 lies in $L^1[0,1]$, we conclude that $\|f_n\|_1\rightarrow 0$, as claimed. However $J$ is not $uaw$-compact, since the norm bounded sequence $(r_n)$ of the Rademacher's functions does not have any $uaw$-convergent subsequence.
\end{exam}

For the $uaw$-convergence, we have $x_\alpha \xrightarrow{uaw} x$ in Banach lattice $E$ if and only if  $|x_\alpha - x| \xrightarrow{uaw} 0$; see \cite[Lemma 1]{Z}. It allows one to reduce $uaw$-convergence to the $uaw$-convergence of positive nets to zero.

\begin{prop} Let $T\colon E\rightarrow F$ be a positive $uaw$-Dunford-Pettis operator between Banach lattices with $F$ Dedekind complete. Then the Kantorovich-like extension $S\colon E\rightarrow F$ defined via
\begin{center}
$S(y)=\sup \left\{ T(y\wedge y_n)\colon  (y_n)\subseteq E_{+}, y_n\xrightarrow{uaw} 0\right\}$
\end{center}
for $y\in E_{+}$ is again $uaw$-Dunford-Pettis.
\end{prop}
\begin{proof}
Suppose $y,z\in E_{+}$. Then
\begin{center}
$S(y+z)=\sup_{n} \{T((y+z)\wedge \gamma_n)\}\le \sup_{n}\{T(y\wedge \gamma_n)\}+\sup_{n}\{T(z\wedge \gamma_n)\}\le S(y)+S(z),$
\end{center}
in which, $(\gamma_n)$ is a positive sequence that is $uaw$-null. On the other hand,
\[T(y\wedge \alpha_n)+T(z\wedge \beta_n)=T(y\wedge \alpha_n+z\wedge \beta_n)\leq T((y+z)\wedge (\alpha_n+\beta_n))\leq S(y+z),\]
provided that two positive sequences $(\alpha_n),(\beta_n)$ are $uaw$-null so that $S(y)+S(z)\leq S(y+z)$. Therefore, by the Kantorovich extension Theorem \cite[Theorem 1.10]{AB}, $S$ extends to a positive operator. Denote by $S$ the extended operator $S\colon E\rightarrow F.$

We show that $S$ is also $uaw$-Dunford-Pettis. Suppose the norm bounded sequence $(y_n)\subseteq E_{+}$ is $uaw$-null. Therefore, we have
\[\|S(y_n)\|=\|\sup_{m}T(y_n\wedge \alpha_m)\|\leq \|T(y_n)\|\rightarrow 0,\]
in which $(\alpha_m)$ is a positive sequence in $E$ which is convergent to zero in the $uaw$-topology.
\end{proof}

In the next example, we show that adjoint of a non $uaw$-Dunford-Pettis operator can be $uaw$-Dunford-Pettis.
\begin{exam}
Consider the operator $T\colon \ell_1\rightarrow L^2[0,1]$ defined by $T((x_n))=(\sum_{i=1}^{\infty}x_n)\chi_{[0,1]}$ for all $(x_n)\in\ell_2$ where $\chi_{[0,1]}$ denotes the characteristic function of $[0,1].$ The operator $T$ is compact but it is not $uaw$-Dunford-Pettis. Its adjoint $T'\colon L^2[0,1]\rightarrow \ell_{\infty}$ is compact, and hence, it is Dunford-Pettis. By Proposition \ref{2}, we conclude that it is $uaw$-Dunford-Pettis.
\end{exam}

\begin{rem}
One may verify that every positive operator which is dominated by a positive $uaw$-Dunford-Pettis operator is again $uaw$-Dunford-Pettis. Therefore, if $T$ is an operator whose modulus is $uaw$-Dunford-Pettis, it can be easily seen that $T$ is also $uaw$-Dunford-Pettis.
Furthermore, a remarkable Theorem by Kalton-Saab ( \cite[Theorem 5.90]{AB}) asserts that if the range space is order continuous, then we can deduce the former statement in the case of Dunford-Pettis operators. So, this point can be considered as an advantage for $uaw$-Dunford-Pettis operators.
\end{rem}
In this step, we investigate closedness properties of $B_{UDP}(E)$.
\begin{prop}\label{1000}
$B_{UDP}(E)$ is closed subalgebra of $B(E)$.
\end{prop}
\begin{proof}
Suppose $(T_m)$ is sequence of $uaw$-Dunford-Pettis operators which is convergent to the operator $T$. We show that $T$ is also $uaw$-Dunford-Pettis. Assume that $(x_n)$ is a bounded $uaw$-null sequence in $E$. Given any $\varepsilon>0$. There is an $m_0$ such that $\|T_m-T\|< \frac{\varepsilon}{2}$ for each $m>m_0$. Fix an $m>m_0$. For sufficiently large $n$, we have $\|T_m(x_n)\|< \frac{\varepsilon}{2}$. Therefore,
\[\|T(x_n)\|<\|T_m-T\|+\|T_m(x_n)\|< \varepsilon.\]
\end{proof}
The class of all $uaw$-Dunford-Pettis operators is not order closed. Consider the following.
\begin{exam}
Put $E=c_0$. Suppose $P_n$ is the projection on the $n$-th first components. Each $P_n$ is finite rank operator so that Dunford-Pettis. By Proposition \ref{2}, it is $uaw$-Dunford-Pettis. Also, $P_n\uparrow {I}$, where $I$ denotes the identity operator on $E$. But $I$ is not $uaw$-Dunford-Pettis as the standard basis $(e_i)$ is $uaw$-null but not norm convergent to zero.
\end{exam}
\begin{rem}
It is a natural and remarkable question to ask whether $B_{UDP}(E)$ can have a lattice structure or not. This can be reduced as follows. When the modulus of a $uaw$-Dunford-Pettis operator exists and is again $uaw$-Dunford-Pettis?
In general, the answer to this question is not affirmative. Consider \cite[Example 5.6]{AB} which is due to Krengel. Observe that the space $E$ mentioned there, is a Dedekind complete  order continuous Banach lattice whose dual is again order continuous. The operator $T$ is compact so that Dunford-Pettis. By Proposition \ref{2}, it is $uaw$-Dunford-Pettis. The sequence $(\hat{x_n})$ is disjoint so that by \cite[Lemma 2]{Z}, it is $uaw$-null. But as we see in the example $|T|(\hat{x_n})$ can not be norm null.

The handy required tool is being disjoint preserving. Recall that an operator $T$ between vector lattices $E$ and $F$ is called disjoint preserving if $x\bot y$ implies that $Tx\bot Ty$.
\end{rem}
\begin{thm}\label{2000}
Suppose $E$ is a Banach lattice. Then every order bounded disjoint preserving operator which is $uaw$-Dunford-Pettis possesses a modulus which is $uaw$-Dunford-Pettis.
\end{thm}
\begin{proof}
By \cite[Theorem 2.40]{AB}, the modulus of $T$ exists and satisfies at the identity $|T|(x)=|T(x)|$ for each positive element $x\in E$. Suppose $(x_n)$ is a bounded positive net which is $uaw$-null. By the hypothesis, $\|T(x_n)\|\rightarrow 0$. So, $|T|(x_n)$ is also norm null.
\end{proof}
\begin{rem}
Observe that there is no inclusion relation between $uaw$-Dunford-Pettis operators and disjoint preserving ones. The identity operator on $\ell_1$ is certainly disjoint preserving but it is not $uaw$-Dunford-Pettis. Furthermore, consider the operator $T$ on $C[0,1]$ defined via $T(f)=(f(0)+f(1))\bf{1}$. One may verify that $T$ is a compact operator so that Dunford-Pettis. By Proposition \ref{2}, it is $uaw$-Dunford-Pettis but it is not disjoint preserving, as mentioned in \cite[Page 117]{AB}.
\end{rem}
Consider this point that disjoint preserving operators need not be a vector space (see \cite[Page 117]{AB} for more details). Suppose $E$ is an Archimedean vector lattice. Recall that an operator $T$ on $E$ is said to be orthomorphism if it is order bounded and band preserving, simultaneously; equivalently, $T$ is an orthomorphism if it is an order bounded operator in which $x\bot y$ results in $T(x)\bot y$. Every orthomorphism is disjoint preserving. Certainly, $Orth(E)=\{T\in L_b(E), x\bot y \Rightarrow Tx\bot y\}$ is a vector subspace of $L_b(E)$ so that an ordered vector space itself, in which $L_b(E)$ is the space of all order bounded operators on $E$. Moreover, by a notable result due to Bigard, Keimel, Conrad, and Diem (\cite[Theorem 2.43]{AB}), $Orth(E)=\{T\in L_b(E), x\bot y \Rightarrow Tx\bot y\}$, is an Archimedean vector lattice where the lattice structures are provided pointwise. Now, we have the following.
\begin{coro}
Suppose $E$ is a Banach lattice and $Orth_{UDP}(E)$ is the space of all $uaw$-Dunford-Pettis orthomomrphisms on $E$.
Then, $Orth_{UDP}(E)$ is an $AM$-space with respect to the regular norm.
\end{coro}
\begin{proof}
Observe that by \cite[Theorem 4.77]{AB} (due to Wickstead), $Orth(E)$ is an $AM$-space with unit.
Note that by using \cite[Theorem 2.40]{AB}, we conclude that for every $T\in Orth(E)$, $\|T\|=\|T\|_r$. This point compatible with Proposition \ref{1000} convince us that $Orth_{UDP}(E)$ is closed so that an $AM$-space in its own right.
\end{proof}

\end{document}